\documentclass{amsart}
\usepackage{amsmath, amssymb, mathrsfs, verbatim, multirow}
\usepackage[shortlabels]{enumitem} 
\usepackage[all]{xy}

\newtheorem{Teo}{Theorem}[section]
\newtheorem{Prop}[Teo]{Proposition}
\newtheorem{Lema}[Teo]{Lemma}
\newtheorem{Cor}[Teo]{Corollary}

\theoremstyle{definition}
\newtheorem{Def}[Teo]{Definition}

\newtheorem{Obs}[Teo]{Remark}

\newtheorem{Exa}[Teo]{Example}

\newcommand{\Q}{\mathbb{Q}}

\newcommand{\N}{\mathbb{N}}

\newcommand{\Llr}{\Longleftrightarrow}
\newcommand{\lra}{\longrightarrow}

\newcommand{\MI}{\mathfrak{m}}

\newcommand{\QF}{\mbox{\rm Quot}}

\newcommand{\inv}{\mbox{\rm in}}

\begin{document}
\title[Generating sequence]{Generating sequences and key polynomials}
\author{M. S. Barnab\'e}
\author{J. Novacoski}
\thanks{During the realization of this project Novacoski was supported by a research grant from Funda\c c\~ao de Amparo \`a Pesquisa do Estado de S\~ao Paulo (process number 2017/17835-9).}

\begin{abstract}
The main goal of this paper is to study the different definitions of generating sequences appearing in the literature. We present these definitions and show that under certain situations they are equivalent. We also present an example that shows that they are not, in general, equivalent. We also present the relation of generating sequences and key polynomials.
\end{abstract}

\maketitle

\section{Introduction}
Given a valuation $\nu$ on a ring $R$, a generating sequence is a subset $\textbf{Q}$ of $R$ that completely determines the valuation $\nu$. The formalization of this idea appears in different works in slightly different ways. One of the main goals of this paper is to stablish the relation between these different definitions.

The concept of graded algebra is closely related to generating sequences. Essentially, a generating sequence is a set whose images generate the graded algebra. Graded algebras play an important role to understand extensions of valuations (see for instance, \cite{Her_1}, \cite{Her_2}, \cite{Vaq_1} and \cite{Vaq_2}). Graded algebras are also central objects in the approach of Teissier for local uniformization (see \cite{Tei} and \cite{Tei2}).

Let $R$ be a ring and $\nu$ a valuation on $R$. We denote by $R^\times$  the set of units of $R$. When $R=K$ is a field, we denote by $K\nu$ the residue field and by $z\nu$ the image in the residue field of $z \in \mathcal{O}_{\nu}$. We denote by $\N_0$ the set of non-negative integers. For $\textbf{Q}\subseteq R$ we will denote by
\[
\N_0^{\textbf{Q}}=\{\lambda:\textbf{Q}\lra \N_0\mid\lambda(Q)\neq 0\mbox{ for only finitely many }Q\in\textbf{Q}\}
\]
and
\[
\textbf{Q}^\lambda=\prod_{\lambda(Q)\neq 0}Q^{\lambda(Q)}\in R.
\]
For each $\gamma\in \nu(R)$, we consider the sets
\begin{equation*}
\mathcal{P}_{\gamma}=\{f \in R \mid \nu(f) \geq \gamma\} \mbox{ and } \mathcal{P}^{+}_{\gamma}=\{f \in R \mid \nu(f) > \gamma\}.
\end{equation*}
The \textbf{graded ring of $R$ associated to $\nu$} is defined as
\begin{equation*}
{\rm gr}_\nu(R)= \displaystyle  \bigoplus_{\gamma \in \nu(R)} \mathcal{P}_{\gamma}/\mathcal{P}^{+}_{\gamma}.
\end{equation*}
The addition on ${\rm gr}_\nu(R)$ is given by the abelian group structure and the multiplication is given explicitly by
\[
\left(f+\mathcal{P}_{\nu(f)}^+\right)\cdot\left(g+\mathcal{P}_{\nu(g)}^+\right):=\left(fg+\mathcal{P}_{\nu(f)+\nu(g)}^+\right)
\]
and extending it to ${\rm gr}_\nu(R)$ in the obvious way.

In \cite{Her_2}, the definition of generating sequence is a slight variation of the following (see discussion in Section \ref{secsequnckeypolgen}). A set $\textbf{Q}\subseteq R$ satisfies \textbf{(GS1)} if for every $\gamma\in\nu(R)$ the abelian group $\mathcal P_\gamma$ is generated by
\[
\left\{a\textbf{Q}^\lambda\mid \nu(a\textbf{Q}^\lambda)\geq \gamma\mbox{ where }\lambda\in\N_0^\textbf{Q} \mbox{ and }a\in R^\times \right\}.
\]

If $\nu$ is centered on $R$ (i.e., $\nu(f)\geq 0$ for every $f\in R$), then $\mathcal{P}_{\gamma}$ and $\mathcal{P}^{+}_{\gamma}$ are ideals of $R$. Moreover, if we set
\[
\MI:=\mathcal P_0^+=\{f\in R\mid \nu(f)>0\},
\]
then $\MI$ is a prime ideal and for each $\gamma\in\nu(R)$ we have that $\mathcal P_\gamma/\mathcal P_\gamma^+$ is an $R/\MI$-module. In this case, ${\rm gr}_\nu(R)$ is an $R/\MI$-algebra, that will be called the \textbf{graded algebra of $R$ associated to $\nu$}. The definition of generating sequence in \cite{Ghes} and \cite{Spi_2} is the following. If $\nu$ is centered at $R$, then $\textbf{Q}\subseteq R$ satisfies \textbf{(GS2)} if for every $\gamma\in\nu(R)$ the ideal $\mathcal P_\gamma$ is generated by
\[
\left\{\textbf{Q}^\lambda\mid \nu(\textbf{Q}^\lambda)\geq \gamma\mbox{ where }\lambda\in\N_0^\textbf{Q}\right\}.
\]

We observe that if $\nu$ is centered, then \textbf{(GS1)} implies \textbf{(GS2)} and that \textbf{(GS2)} only makes sense for centered valuations. Moreover, when dealing with key polynomials for a valuation $\nu$ on $K[x]$, the case when $\nu$ is centered is not interesting (see Lemma \ref{keypolnotinterscenters}).

The definition of a generating sequence in \cite{Cutgen}, \cite{Mour} and \cite{Vihn} is the following. If $\nu$ is centered at $R$, then $\textbf{Q}\subseteq R$ satisfies \textbf{(GS3)} if the set
\[
{\rm in}_\nu(\textbf{Q}):=\{{\rm in}_\nu(Q)\mid Q\in\textbf{Q}\}
\]
generates ${\rm gr}_\nu(R)$ as an $R/\MI$-algebra.

One of the main goals of this paper is to stablish the relation between these different definitions. More specifically, we prove the following.
\begin{Teo}\label{maintheomr}
Let $R$ be a ring, $\nu$ a valuation on $R$ and $\textbf{Q}$ a subset of $R$. 
\begin{description}
\item[(i)] If $\nu$ is centered and \textbf{Q} satisfies \textbf{(GS2)}, then for every $\gamma\in\nu(R)$ the group $\mathcal P_\gamma/\mathcal P_\gamma^+$ is generated by 
\[
\left\{a\textbf{Q}^\lambda+\mathcal P_\gamma^+\mid \lambda\in\N_0^\textbf{Q},\nu(\textbf{Q}^\lambda)=\gamma\mbox{ and }\nu(a)=0\right\}.
\]
In particular, \textbf{Q} satisfies \textbf{(GS3)}.

\item[(ii)] If \textbf{Q} satisfies \textbf{(GS3)}, then the semigroup $\nu(R)$ is generated by
\[
\nu(\textbf{Q}):=\{\nu(Q)\mid Q\in\textbf{Q}\}.
\]
Conversely, if $R$ is a domain with $K=\QF(R)$, $K\nu=R/\MI$ and $\nu(\textbf{Q})$ generates $\nu(R)$, then \textbf{(GS3)} is satisfied.
\end{description}
\end{Teo}

In Section \ref{proofomaintheom} we present an example that shows that the converse of Theorem \ref{maintheomr} \textbf{(i)} does not hold in general.

We also study the relation of generating sequences and key polynomials for valuations $\nu$ defined on $K[x]$. The interesting case of study for key polynomials is for valuations which are not trivial on $K$. Hence, we want to compare sequences of key polynomials which are complete and sequences satisfying \textbf{(GS1)}. We make a change on \textbf{(GS1)} when $R=K[x]$. We will say that $\textbf{Q}$ satisfies \textbf{(GS1$^*$)} if for every $f\in K[x]$ there exist $a_1,\ldots,a_r\in K$ and $\lambda_1,\ldots, \lambda_r\in \N_0^\textbf{Q}$, such that
\[
f=\sum_{i=1}^ra_i\textbf{Q}^{\lambda_i}\mbox{ with }\nu\left(a_i\textbf{Q}^{\lambda_i}\right)\geq \nu(f),\mbox{ for every }i, 1\leq i\leq r,
\]
and for every $i$, $1\leq i\leq r$, if $\lambda_i(Q)\neq 0$, then $\deg(Q)\leq\deg(f)$. Since $K[x]^\times=K^\times$, the only difference between \textbf{(GS1)} and \textbf{(GS1$^*$)} is the condition on the degrees.

Another important result of this paper is the following.
\begin{Teo}\label{theormjepolgensequ}
Let $\textbf{Q}$ be a set of key polynomials for $K[x]$. Then $\textbf{Q}$ is complete if and only if \textbf{Q} satisfies \textbf{(GS1$^*$)}.
\end{Teo}

This paper is divided as follows. In Section \ref{preliminaires} we present a few basic results about graded algebras. In Section \ref{proofomaintheom} we present the proof of Theorem \ref{maintheomr} as well as the example that shows that its converse is not satisfied. In Section \ref{keypolynomials} we present the definition and the main results about key polynomials (as in \cite{SopivNova}). We also present some results that do not appear in \cite{SopivNova} but will be needed here. Finally, in Section \ref{secsequnckeypolgen} we present the relation between key polynomials and generating sequences.

\textbf{Acknowledgements.} We would like to thank the anonymous referee for a careful reading and for pointing out a few mistakes in an earlier version of this paper.

\section{Preliminaries}\label{preliminaires}
Let $R$ be a ring and $\nu$ a valuation on $R$. For $f\in R$ we will denote by ${\rm in}_\nu(f)$ the image of $f$ in
\[
\mathcal P_{\nu(f)}/\mathcal P_{\nu(f)}^+\subseteq{\rm gr}_\nu(R).
\]
We have the following properties in ${\rm gr}_\nu(R)$.
\begin{Lema}\label{basicpropestisgradealg}
Take $f,g\in R$.
\begin{description}
\item[(i)] $\inv_\nu(f)\cdot \inv_\nu(g)=\inv_\nu(fg)$.
\item[(ii)] $\inv_\nu(f)=\inv_\nu(g)$ if and only if $\nu(f)=\nu(g)\mbox{ and }\nu(f-g)>\nu(f)$.
\item[(iii)] If $\nu(f)<\nu(g)$, then $\inv_\nu(f)=\inv_\nu(f+g)$.
\item[(iv)] $\inv_\nu(f+g)=\inv_\nu(f)+\inv_\nu(g)$ if and only if $\nu(f)=\nu(g)=\nu(f+g)$.
\end{description}
\end{Lema}
\begin{proof}
Item \textbf{(i)} follows directly from the definition. For \textbf{(ii)} we observe that if $\nu(f)\neq \nu(g)$, then $\inv_\nu(f)\neq \inv_\nu(g)$ because they are in different components of ${\rm gr}_\nu(R)$. Moreover, in the case $\nu(f)= \nu(g)$, we have that $\inv_\nu(f)=\inv_\nu(g)$ means that $(f-g)\in\mathcal P_{\nu(f)}^+$, which means that $\nu(f-g)>\nu(f)$.

Assume that $\nu(f)<\nu(g)$. Then $\nu(f+g)=\nu(f)<\nu(g)$, hence by item \textbf{(ii)} we have $\inv_\nu(f)=\inv_\nu(f+g)$. In order to prove \textbf{(iv)}, we observe that if $\nu(f)=\nu(g)=\nu(f+g)$, then
\begin{displaymath}
\begin{array}{rcl}
\inv_\nu(f+g)&=&(f+g)+\mathcal P_{\nu(f)}^+=(f+\mathcal P_{\nu(f)}^+)+(g+\mathcal P_{\nu(g)}^+)\\[8pt]
&=&\inv_\nu(f)+\inv_\nu(g).
\end{array}
\end{displaymath}
The converse follows directly from \textbf{(iii)}.
\end{proof}

\begin{Lema}\label{anotherlemasobregradeqseque}
Assume that
\[
\inv_\nu(f)=\inv_\nu\left(\sum_{i=1}^rf_i\right)\mbox{ for some }f_1,\ldots,f_r\in R,
\]
with $\nu(f)=\nu(f_i)$ for every $i$, $1\leq i\leq r$. Then there exists $I\subseteq \{1,\ldots,r\}$ such that
\[
\inv_\nu(f)=\sum_{i\in I}\inv_\nu\left(f_i\right).
\]
\end{Lema}
\begin{proof}
We will prove by induction on $r$. If $r=1$, there is nothing to prove. Assume that $r>1$ and that the result holds for $r-1$. If $\displaystyle\nu(f_1)<\nu\left(\sum_{i=2}^rf_i\right)$, then by Lemma \ref{basicpropestisgradealg} \textbf{(iii)} we have
\[
\inv_\nu(f_1)=\inv_\nu\left(f_1+\sum_{i=2}^rf_i\right)=\inv_\nu(f).
\]
On the other hand, if $\displaystyle\nu(f_1)=\nu\left(\sum_{i=2}^rf_i\right)$, then by Lemma \ref{basicpropestisgradealg} \textbf{(iv)} we have
\[
\inv_\nu(f)=\inv_\nu\left(f_1+\sum_{i=2}^rf_i\right)=\inv_\nu(f_1)+\inv_\nu\left(\sum_{i=2}^rf_i\right)
\]
and the result follows by the induction hypothesis.
\end{proof}

\section{Proof of Theorem \ref{maintheomr}}\label{proofomaintheom}
\begin{proof}[Proof of Theorem \ref{maintheomr}]
In order to prove \textbf{(i)}, assume that \textbf{Q} satisfies \textbf{(GS2)}. 
Take $\gamma\in\nu(R)$ and choose $f\in R$ such that $\nu(f)= \gamma$. Since $\mathcal P_\gamma$ is generated as an ideal by
\[
\left\{\textbf{Q}^\lambda\mid \lambda\in\N_0^\textbf{Q}\mbox{ and }\nu\left(\textbf{Q}^\lambda\right)\geq \gamma\right\}
\]
there exist $a_1,\ldots,a_r\in R$ and $\lambda_1,\ldots,\lambda_r\in \N_0^{\textbf{Q}}$ such that
\begin{equation}\label{equaajudanaproga}
f=\sum_{i=1}^ra_i\textbf{Q}^{\lambda_i}\mbox{ and }\nu\left(\textbf{Q}^{\lambda_i}\right)\geq \gamma=\nu(f)\mbox{ for every }i,1\leq i\leq r.
\end{equation}
Let
\[
I=\{i\in\{1,\ldots,r\}\mid \nu(a_i)=0\mbox{ and }\nu\left(\textbf{Q}^{\lambda_i}\right)=\gamma\}.
\]
It follows from \eqref{equaajudanaproga} that $I\neq\emptyset$. Then
\[
\nu\left(f-\sum_{i\in I}a_i\textbf{Q}^{\lambda_i}\right)=\nu\left(\sum_{i\notin I}a_i\textbf{Q}^{\lambda_i}\right)>\gamma.
\]
Therefore, $f-\displaystyle\sum_{i\in I}a_i\textbf{Q}^{\lambda_i}\in \mathcal P_\gamma^+$. Moreover, in this case (using Lemma \ref{anotherlemasobregradeqseque})
\[
{\rm in}_\nu(f)=\inv_\nu\left(\sum_{i\in I}a_i\textbf{Q}^{\lambda_i}\right)=\sum_{i\in I'}\inv_\nu\left(a_i\textbf{Q}^{\lambda_i}\right)=\sum_{i\in I'}\overline{a_i}\left(\inv_\nu\left(\textbf{Q}\right)\right)^{\lambda_i}
\]
for some $I'\subseteq I$. Therefore, \textbf{Q} satisfies \textbf{(GS3)}.

In order to prove \textbf{(ii)} assume that $\textbf{Q}\subseteq R$ satisfies \textbf{(GS3)}. Take $\gamma\in \nu(R)$ and $f\in R$ such that $\nu(f)=\gamma$. By our assumption, there exist $a_1,\ldots,a_r\in R\setminus \MI$ and $\lambda_1,\ldots,\lambda_r\in \N_0^{\textbf{Q}}$ such that
\begin{equation}\label{esobrerelasgendseq}
\inv_\nu(f)=\sum_{i=1}^r\overline{a_i}\inv_\nu\left(\textbf{Q}^{\lambda_i}\right).
\end{equation}
Since $\inv_\nu(f)$ is homogeneous, we can assume that the elements on the right of \eqref{esobrerelasgendseq} are homogeneous and are in the same homogeneous component as $\inv_\nu(f)$. This means that for each of such $i$ we have
\[
\gamma=\nu(f)=\nu\left(\textbf{Q}^{\lambda_i}\right)=\sum_{\lambda_i(Q)\neq 0}\lambda_i(Q)\nu(Q),
\]
hence $\nu(R)$ is generated by $\nu(\textbf{Q})$.

Now assume that $K\nu=R/\MI$ and that $\nu(R)$ is generated by $\nu(\textbf{Q})$. For $f\in R$ there exist $n_1,\ldots,n_r\in\N$ and $Q_1,\ldots,Q_r\in \textbf{Q}$ such that
\[
\nu(f)=\sum_{i=1}^rn_i\nu(Q_i)=\nu\left(\prod_{i=1}^rQ_i^{n_i}\right)
\]
Since $K\nu=R/\MI$, there exists $z\in R\setminus \MI$ such that
\[
z\nu=\frac{f}{\displaystyle\prod_{i=1}^rQ_i^{n_i}}\nu.
\]
This means that $\nu\left(f-\displaystyle z\prod_{i=1}^rQ_i^{n_i}\right)>\nu(f)$ and consequently
\[
\inv_\nu(f)=\inv_\nu\left(z\prod_{i=1}^rQ_i^{n_i}\right)=\overline{z}\prod_{i=1}^r\left(\inv_\nu(Q_i)\right)^{n_i}.
\]
\end{proof}

\begin{Exa}
This is to show that the converse of Theorem \ref{maintheomr} \textbf{(i)} does not hold in general. Consider a field $k$ and the valuation on $k[x,y]$ induced by the embedding of $k[x,y]$ in $k((t^\Q))$ defined by
\[
x\longmapsto t\mbox{ and }y\longmapsto\sum_{i=1}^\infty t^{i^2}=t+t^4+t^9+t^{16}+\ldots,
\]
where $\nu(k^\times)=0$ and $\nu(t)=1$. For every $p(x,y)\in k[x,y]$ there exists $n_0\in\N$ and $a_{n_0}\in k$ such that
\[
\nu(p(x,y)-a_{n_0}x^{n_0})>n_0.
\]
Indeed, we can write
\[
p\left(t,\sum_{i=1}^\infty t^{i^2}\right)=\sum_{k=n_0}^\infty a_kt^k.
\]
Hence,
\[
\nu\left(p(x,y)-a_{n_0}x^{n_0}\right)=\nu_t\left(p\left(t,\sum_{i=1}^\infty t^{i^2}\right)-a_{n_0}t^{n_0}\right)>n_0.
\]
Since $\nu\left(a_{n_0}x^{n_0}\right)=n_0$ we have that
\[
\inv_\nu(p)=\inv_\nu(a_{n_0}x^{n_0})=\overline{a_{n_0}}\inv_\nu(x)^{n_0}.
\]
Therefore, $\textbf{Q}=\{x\}$ satisfies \textbf{(GS3)}.

However, $x$ and $y$ belong to ideal $\mathcal P_{\gamma}$, where $\gamma=1$, and since $x$ and $y$ are algebraically independent over $k$, for every $p_1,\ldots,p_r\in k[x,y]$ we have
\[
y\neq\sum_{i=1}^rp_ix^i.
\]
Therefore, $\textbf{Q}=\{x\}$ does not satisfy \textbf{(GS2)}.
\end{Exa}

\section{Key polynomials}\label{keypolynomials}
In order to define a key polynomial, we will need to define the number $\epsilon(f)$ for $f\in K[x]$. Let $\Gamma'=\Gamma\otimes \Q$ be the divisible hull of $\Gamma$. For a polynomial $f\in K[x]$ and $k\in\N$, we consider
\[
\partial_k(f):=\frac{1}{k!}\frac{d^kf}{dx^k},
\]
the so called Hasse-derivative of $f$ of order $k$. Let
\[
\epsilon(f)=\max_{k\in \N}\left\{\frac{\nu(f)-\nu(\partial_kf)}{k}\right\}\in \Gamma'.
\]
\begin{Def}
A monic polynomial $Q\in K[x]$ is said to be a \textbf{key polynomial} (of level $\epsilon (Q)$) if for every $f\in K[x]$ if $\epsilon(f)\geq \epsilon(Q)$, then $\deg(f)\geq\deg(Q)$.
\end{Def}

The next result is a characterization for $\epsilon(f)$. Consider an extension $\mu$ of $\nu$ to $\overline K[x]$ (here $\overline K$ denotes an algebraic closure of $K$). For a polynomial $f\in K[x]$, we define
\[
\delta(f)=\max\{\mu(x-a)\mid a\mbox{ is a root of }f\}.
\]
\begin{Prop}[Proposition 3.1 of \cite{Nov}]
If $f\in K[x]$ is a monic polynomial, then $\epsilon(f)=\delta(f)$.
\end{Prop}
\begin{Obs}
The condition of being monic can be dropped in the proposition above. This follows from the fact that for every $c\in K^\times$ we have that
\[
\epsilon(f)=\max_{k\in \N}\left\{\frac{\nu(f)-\nu(\partial_kf)}{k}\right\}=\max_{k\in \N}\left\{\frac{\nu(cf)-\nu(\partial_k(cf))}{k}\right\}=\epsilon(cf)
\]
and since $f$ and $cf$ have the same roots, we have $\delta(f)=\delta(cf)$.
\end{Obs}
\begin{Cor}
For $f,g\in K[x]$ we have $\epsilon(fg)=\max\{\epsilon(f),\epsilon(g)\}$.
\end{Cor}
\begin{proof}
We have that $a$ is a root of $fg$ if and only if it is a root of $f$ or $g$. Hence
\begin{displaymath}
\begin{array}{rcl}
\epsilon(fg)&=&\delta(fg)=\max\{\mu(x-a)\mid a\mbox{ is a root of }fg\}\\[8pt]
&=&\max\{\mu(x-a)\mid a\mbox{ is a root of }f\mbox{ or }g\}=\max\{\delta(f),\delta(g)\}\\[8pt]
&=&\max\{\epsilon(f),\epsilon(g)\}.
\end{array}
\end{displaymath}
\end{proof}
\begin{Cor}\label{proposdolivrokeypolirredu}
Every key polynomial is irreducible.
\end{Cor}
\begin{proof}
Assume that $Q$ is a key polynomial and suppose it is not irreducible. Write $Q=fg$ where $\deg(f)<\deg(Q)$ and $\deg(g)<\deg(Q)$. Then by the previous result $\epsilon(Q)=\max\{\epsilon(f),\epsilon(g)\}$. This implies that $\epsilon(f)=\epsilon(Q)$ or $\epsilon(g)=\epsilon(Q)$. Since $\deg(f)<\deg(Q)$ and $\deg(g)<\deg(Q)$ this is a contradiction to the fact that $Q$ is a key polynomial.
\end{proof}
\begin{Obs}
Corollary \ref{proposdolivrokeypolirredu} was proved in \cite{SopivNova} (Proposition 2.4 \textbf{(ii)}). However, the proof above is simpler.
\end{Obs}

Given two polynomials $f, Q \in K[x]$ with $Q$ monic, we call $Q$-{\bf expansion of} $f$ the expression
$$ f= f_0 + f_1Q + \ldots + f_nQ^n,$$
where for each $i$, $0 \leq i \leq n$, $f_i=0$ or $\deg(f_i)<\deg(Q)$. For a monic polynomial $Q \in K[x]$, the $Q$-{\bf truncation of} $\nu$ is defined as
$$ \nu_Q(f):= \min_{0 \leq i \leq n} \{ \nu(f_iQ^i)\},$$
where $ f= f_0 + f_1Q + \ldots + f_nQ^n$ is the $Q$-expansion of $f$.

\begin{Lema}[Lemma 2.3 of \cite{SopivNova}]
Let $Q$ be a key polynomial and take $f,g\in K[x]$ such that
\[
\deg(f)<\deg(Q)\mbox{ and }\deg(g)<\deg(Q).
\]
Then for $\epsilon:=\epsilon(Q)$ and any $k\in\N$ we have the following:
\begin{description}\label{lemaonkeypollder}
\item[(i)] $\nu(\partial_k(fg))>\nu(fg)-k\epsilon$
\item[(ii)] If $\nu_Q(fQ+g)<\nu(fQ+g)$ and $k\in I(Q):=\left\{i\mid \epsilon(f)=\frac{\nu(f)-\nu(\partial_if)}{i}\right\}$, then $\nu(\partial_k(fQ+g))=\nu(fQ)-k\epsilon$;
\item[(iii)] If $h_1,\ldots,h_s$ are polynomials such that $\deg(h_i)<\deg(Q)$ for every $i=1,\ldots, s$ and
$\displaystyle\prod_{i=1}^sh_i=qQ+r$ with $\deg(r)<\deg(Q)$ and $r\neq 0$, then
\[
\nu(r)=\nu\left(\prod_{i=1}^sh_i\right)<\nu(qQ).
\]
\end{description}
\end{Lema}

\begin{Prop}[Proposition 2.6 of \cite{SopivNova}]\label{proptruncakeypolval}
If $Q$ is a key polynomial, then $\nu_Q$ is a valuation of $K[x]$.
\end{Prop}

\begin{Prop}[Proposition 2.10 of \cite{SopivNova}]\label{Propcompkeypol}
For two key polynomials $Q,Q'\in K[x]$ we have the following:
\begin{description}
\item[(i)] If $\deg(Q)<\deg(Q')$, then $\epsilon(Q)<\epsilon(Q')$;
\item[(ii)] If $\epsilon(Q)<\epsilon(Q')$, then $\nu_Q(Q')<\nu(Q')$;
\item[(iii)] If $\deg(Q)=\deg(Q')$, then
\begin{equation}\label{eqwhdegsame}
\nu(Q)<\nu(Q')\Llr \nu_Q(Q')<\nu(Q')\Llr \epsilon(Q)<\epsilon(Q').
\end{equation}
\end{description}
\end{Prop}
\begin{Cor}\label{corsobreequaldepsoidkey}
Let $Q$ and $Q'$ be key polynomials such that $\epsilon(Q)\leq\epsilon(Q')$. For every $f\in K[x]$, if $\nu_Q(f)=\nu(f)$, then $\nu_{Q'}(f)=\nu(f)$.
\end{Cor}
\begin{proof}
It follows from Proposition \ref{Propcompkeypol} that if $\epsilon(Q)\leq\epsilon(Q')$, then $\nu_{Q'}(Q)=\nu(Q)$. Since $\deg(Q)\leq \deg(Q')$, for every $f_i\in K[x]$ with $\deg(f_i)<\deg(Q)$ we have $\nu_{Q'}(f_i)=\nu(f_i)$. Hence $\nu_{Q'}(f_iQ^i)=\nu(f_iQ^i)$.

Take $f\in K[x]$ such that $\nu_Q(f)=\nu(f)$ and let
\[
f=f_0+f_1 Q+\ldots+f_nQ^n
\]
be the $Q$-expansion of $f$. Then
\[
\nu_{Q'}(f)\geq\min_{0\leq i\leq n}\{\nu_{Q'}(f_iQ^i)\}=\min_{0\leq i\leq n}\{\nu(f_iQ^i)\}=\nu_Q(f)=\nu(f).
\]
Since $\nu_{Q'}(f)\leq \nu(f)$ for every $f\in K[x]$ we have our result.
\end{proof}

\begin{Def}\label{definitioncomplsetes}
A set $\textbf{Q}\subseteq K[x]$ is called a \textbf{complete set for $\nu$} if for every $f\in K[x]$ there exists $Q\in \textbf{Q}$ with $\deg(Q)\leq\deg(f)$ such that $\nu_Q(f)=\nu(f)$. If the set \textbf{Q} admits an order under which it is well-ordered, then it is called a complete sequence.
\end{Def}

\begin{Teo}[Theorem 1.1 of \cite{SopivNova}]\label{Theoremexistencecompleteseqkpol}
Every valuation $\nu$ on $K[x]$ admits a complete set $\textbf{Q}$ of key polynomials. Moreover, $\textbf{Q}$ can be chosen to be well-ordered with respect to the order given by $Q<Q'$ if $\epsilon(Q)<\epsilon(Q')$.
\end{Teo}
\begin{Obs}
In \cite{SopivNova}, the definition of complete sequence does not require that $\deg(Q)\leq\deg(f)$ as in Definition \ref{definitioncomplsetes} above. This property is important and the proof of Theorem 1.1 in \cite{SopivNova} guarantees that the obtained sequence satisfies the additional property.
\end{Obs}

\section{Generating sequences vs key polynomials}\label{secsequnckeypolgen}
In this section we will discuss the relation between key polynomials and generating sequences. We start with the following easy lemma.
\begin{Lema}\label{keypolnotinterscenters}
A valuation $\nu$ on $K[x]$ is centered if and only if $\nu(a)=0$ for every $a\in K\setminus\{0\}$ and $\nu(x)\geq 0$. 
\end{Lema}
\begin{proof}
Assume that $\nu$ is centered. In particular, $\nu(x)\geq 0$. Moreover, since $\nu$ is centered, $\nu(a)\geq 0$ for every $a\in K\setminus\{0\}$. If $\nu(a)>0$, then
\[
\nu(a^{-1})=-\nu(a)<0,
\]
which is a contradiction. Hence, $\nu(a)=0$.

For the converse, assume that $\nu(x)\geq 0$ and $\nu(a)=0$ for every $a\in K\setminus\{0\}$. For every $p(x)=a_0+\ldots+a_nx^n\in K[x]$ we have
\[
\nu(p(x))\geq \min\{\nu(a_ix^i)\}\geq 0.
\]
Hence, $\nu$ is centered.
\end{proof}

In order to prove Theorem \ref{theormjepolgensequ}, we will need a few results. Our first result shows that any complete set (independently of being formed by key polynomials) satisfies \textbf{(GS1$^*$)}.
\begin{Prop}\label{propsobrecomp}
If $\textbf Q\subseteq K[x]$ is a complete set for $\nu$, then \textbf{Q} satisfies \textbf{(GS1$^*$)}.
\end{Prop}
\begin{proof}
We will prove by induction on the degree of $f$. If $\deg(f)=1$, then $f=x-a$ for some $a\in K$. By our assumption, there exists $x-b\in \textbf{Q}$ such that
\[
\beta:=\nu(x-a)=\nu_{x-b}(x-a)=\min\{\nu(x-b),\nu(b-a)\}.
\]
This implies that $\nu(x-b)\geq \beta$, $\nu(b-a)\geq \beta$ and that $f=(x-b)+(b-a)$, which is what we wanted to prove.

Assume now that for $k\in\N$, for every $f\in K[x]$ of $\deg(f)<k$ our result is satisfied. Let $f$ be a polynomial of degree $k$. Since $\textbf{Q}$ is a complete set for $\nu$, there exists $q\in \textbf{Q}$ such that $\deg(q)\leq \deg(f)$ and $\nu_q(f)=\nu(f)$. Let
\[
f=f_0+f_1q+\ldots+f_sq^s
\]
be the $q$-expansion of $f$. Since $\deg(q)\leq \deg(f)$, we have $\deg(f_i)<\deg(f)=k$ for every $i$, $1\leq i\leq s$. By the induction hypothesis, there exist
\[
a_{11},\ldots, a_{1r_1},\ldots,a_{s1},\ldots,a_{sr_s}\in K\mbox{ and }\lambda_{11},\ldots, \lambda_{1r_1},\ldots,\lambda_{s1},\ldots,\lambda_{sr_s}\in \N_0^\textbf{Q},
\]
such that for every $i$, $0\leq i\leq s$,
\[
f_i=\sum_{j=1}^{r_i}a_{ij}\textbf{Q}^{\lambda_{ij}}\mbox{ with }\nu\left(a_{ij}\textbf{Q}^{\lambda_{ij}}\right)\geq \nu(f_i)\mbox{ for every }j, 1\leq j\leq r_i,
\]
and $\deg(Q)\leq \deg(f_i)\leq \deg(f)$ for every polynomial $Q$ for which $\lambda_{ij}(Q)\neq 0$ for some $i,j$. This implies that
\[
f=\sum_{i=0}^s\left(\sum_{j=1}^{r_i}a_{ij}\textbf{Q}^{\lambda_{ij}}\right)q^i=\sum_{0\leq i\leq s, 1\leq j\leq r_i}a_{ij}\textbf Q^{\lambda_{ij}'},
\]
where
\begin{displaymath}
\lambda'_{ij}(q')=\left\{
\begin{array}{ll}
\lambda_{ij}(q')+i&\mbox{ if }q'=q\\
\lambda_{ij}(q')&\mbox{ if }q'\neq q
\end{array}
\right..
\end{displaymath}
Moreover, since $\nu_q(f)=\displaystyle\min_{0\leq i\leq s}\{\nu(f_iq^i)\}=\nu(f)$ and
\[
\nu\left(a_{ij}\textbf{Q}^{\lambda_{ij}}\right)\geq \nu(f_i),\mbox{ for every }i, 0\leq i\leq n\mbox{ and }j, 1\leq j\leq r_i,
\]
we have
\[
\nu(f)\leq \nu(f_i)+i\nu(q)\leq \nu\left(a_{ij}\textbf{Q}^{\lambda_{ij}}\right)+i\nu(q)=\nu\left(a_{ij}\textbf{Q}^{\lambda'_{ij}}\right),
\]
for every $i$, $0\leq i\leq s$ and $j$, $1\leq j\leq r_i$, which is what we wanted to prove.
\end{proof}

The next result gives a converse for Proposition \ref{propsobrecomp}.

\begin{Prop}\label{propogs1estrela}
Assume that $\textbf{Q}$ is a subset of $K[x]$ with the following properties:
\begin{itemize}
\item $\nu_Q$ is a valuation for every $Q\in\textbf{Q}$;
\item for every finite subset $\mathcal F\subseteq \textbf{Q}$, there exists $Q'\in \mathcal F$ such that $\nu_{Q'}(Q)=\nu(Q)$ for every $Q\in\mathcal{F}$;
\item \textbf{(GS1$^*$)} is satisfied.
\end{itemize}
Then $\textbf Q$ is a complete set for $\nu$.
\end{Prop}
\begin{proof}
Take any polynomial $f\in K[x]$ and let $\beta:=\nu(f)$. Then, there exist $a_1,\ldots,a_r\in K$ and $\lambda_1,\ldots,\lambda_r\in\N_0^\textbf{Q}$ such that 
\[
f=\sum_{i=1}^ra_i\textbf{Q}^{\lambda_i}\mbox{ with }\nu\left(a_i\textbf{Q}^{\lambda_i}\right)\geq \beta,\mbox{ for every }i, 1\leq i\leq r,
\]
and $\deg(Q)\leq\deg(f)$ for every $Q\in \textbf{Q}$ for which $\lambda_i(Q)\neq 0$ for some $i$, $1\leq i\leq r$.
Let
\[
\mathcal F:=\{Q\in \textbf{Q}\mid \lambda_i(Q)\neq 0\mbox{ for some }i,1\leq i\leq n\}.
\]
Since $\mathcal F$ is finite, there exists $Q'\in\mathcal F$ such that $\nu_{Q'}(Q)=\nu(Q)$ for every $Q\in \mathcal F$. In particular, $\nu\left(a_i\textbf{Q}^{\lambda_i}\right)=\nu_{Q'}\left(a_i\textbf{Q}^{\lambda_i}\right)$ for every $i$, $1\leq i\leq r$. Then
\[
\beta\leq \min_{1\leq i\leq n}\left\{\nu\left(a_i\textbf Q^{\lambda_i}\right)\right\}=\min_{1\leq i\leq n}\left\{\nu_{Q'}\left(a_i\textbf Q^{\lambda_i}\right)\right\}\leq \nu_{Q'}(f)\leq \nu(f)=\beta.
\]
Therefore, $\nu_Q(f)=\nu(f)$ and this concludes the proof.
\end{proof}

\begin{proof}[Proof of Theorem \ref{theormjepolgensequ}]
If $\textbf{Q}$ is a complete set for $\nu$, then by Proposition \ref{propsobrecomp} \textbf{Q} satisfies \textbf{(GS1$^*$)}.

To prove the converse, we observe that since every element $Q$ in \textbf{Q} is a key polynomial, by Proposition \ref{proptruncakeypolval}, we have that $\nu_Q$ is a valuation. Moreover, since \textbf{Q} is ordered by $Q<Q'$ if and only if $\epsilon(Q)<\epsilon(Q')$, for every finite set $\mathcal F$ we can choose $Q'\in\mathcal F$ such that $\epsilon(Q)\leq\epsilon(Q')$ for every $Q\in\mathcal F$. Applying Corollary \ref{corsobreequaldepsoidkey} we obtain that
\[
\nu_{Q'}(Q)=\nu(Q)\mbox{ for every }Q\in \mathcal F.
\]
The result now follows from Proposition \ref{propogs1estrela}.
\end{proof}

An interesting consequence of Theorem \ref{theormjepolgensequ} is the following.

\begin{Cor}
For every valuation $\nu$ on $K[x]$, there exists a set of key polynomials $\textbf{Q}\subseteq K[x]$ such that \textbf{Q} satisfies \textbf{(GS1*)}. Moreover, this set can be chosen to be well-ordered with respect to the order $Q<Q'$ if $\epsilon(Q)<\epsilon(Q')$.
\end{Cor}
\begin{proof}
It follows directly from Theorems \ref{Theoremexistencecompleteseqkpol} and \ref{theormjepolgensequ}.
\end{proof}

\noindent{\footnotesize MATHEUS DOS SANTOS BARNAB\'E\\
Departamento de Matem\'atica--UFSCar\\
Rodovia Washington Lu\'is, 235\\
13565-905 - S\~ao Carlos - SP\\
Email: {\tt matheusbarnabe@dm.ufscar.br} \\\\

\noindent{\footnotesize JOSNEI NOVACOSKI\\
Departamento de Matem\'atica--UFSCar\\
Rodovia Washington Lu\'is, 235\\
13565-905 - S\~ao Carlos - SP\\
Email: {\tt josnei@dm.ufscar.br} \\\\

\end{document}